\documentclass[]{amsart}
\usepackage{amsmath, amsfonts,amsthm,amssymb,mathrsfs}
\usepackage{enumerate}
\usepackage[dvipdfmx]{graphicx}
\usepackage{ascmac}
\usepackage[arrow,matrix]{xy}
\usepackage[dvips]{color} 
\usepackage{longtable}
\usepackage{mathtools}

\newtheorem{theo}{Theorem}

\newtheorem{defi}{Definition}
\newtheorem{exa}{Example}
\newtheorem{prop}{Proposition}
\newtheorem{cor}{Corollary}

\renewcommand{\proofname}{\textit{Proof.}}
\renewcommand{\qedsymbol}{$\square$}

\makeatletter
    \renewcommand{\theequation}{%
    \thesection.\arabic{equation}}
    \@addtoreset{equation}{section}
  \makeatother

\def\Z{\mathbb{Z}}
\def\C{\mathbb{C}}

\def\N{\mathbb{N}}
\def\H{\mathbb{H}}
\def\R{\mathbb{R}}

\def\e{\equiv}

\begin{document}

\title[Growth rates of cocompact hyperbolic Coxeter groups]{Growth rates of cocompact hyperbolic Coxeter groups and $2$-Salem numbers}
\author{Yuriko Umemoto}
\address{Department of Mathematics, Osaka City University, 558-8585, Osaka, Japan}
\email{yuriko.ummt.77@gmail.com}
\date{}

\maketitle

\begin{abstract}
By the results of Cannon, Wagreich and Parry, it is known that the growth rate of a cocompact Coxeter group in $\H^2$ and $\H^3$ is a Salem number.
Kerada defined a $j$-Salem number, which is a generalization of a Salem number.
In this paper, we realize infinitely many $2$-Salem numbers as the growth rates of cocompact Coxeter groups in $\H ^4$.
Our Coxeter polytopes are constructed by successive gluing of Coxeter polytopes which we call Coxeter dominoes. 
\end{abstract}

\section{Introduction}
Let $\H^n$ denote hyperbolic $n$-space.
A Coxeter polytope $P\subset \H^n$ is a convex polytope of dimension $n$ all of whose dihedral angles are of the form $\pi /m$, where $m\geq 2$ is an integer.
By a well known result, the group $W$ generated by reflections with respect to the hyperplanes bounding $P$ is a discrete subgroup of the isometry group of $\H^n$ whose fundamental domain is $P$, and $W$ itself is called a hyperbolic Coxeter group.
If $P$ is compact, $W$  is called cocompact.

As typical quantities related to Coxeter groups, the growth series and the growth rates of them will be  studied.
The growth series is a formal power series (see (\ref{series})), and Steinberg \cite{St} proved that the growth series of an infinite Coxeter group is an expansion of a rational function (see (\ref{st})).
As a result, the growth rate (see (\ref{rate})) which is defined as the radius of convergence of the growth series is an algebraic integer.

The focus of this work is the arithmetic property of the growth rates of cocompact hyperbolic Coxeter groups.
Cannon, Wagreich, and Parry \cite{CW, Pa} proved that the growth rates of cocompact hyperbolic Coxeter groups in $\H^2$ and $\H^3$ are Salem numbers.
Here a Salem number is defined as a real algebraic integer $\alpha>1$ such that $\alpha^{-1}$ is an algebraic conjugate of $\alpha$ and all algebraic conjugates of $\alpha$ other than $\alpha$ and $\alpha^{-1}$ lie on the unit circle.
Their definition includes a quadratic unit as a Salem number while the ordinary definition does not include it (see \S 3.2).
On the other hand, the growth rates of cocompact hyperbolic Coxeter groups in $\H^4$ are not Salem numbers in general (see \cite{KP}).

As a kind of generalization of a Salem number, Kerada \cite{Ker} defined a $j$-Salem number.
In \cite{ZZ}, T. Zehrt and C. Zehrt constructed infinitely many growth functions of cocompact hyperbolic Coxeter groups in $\H^4$ whose denominator polynomials have the same distribution of roots as $2$-Salem polynomials. 
More precisely, their Coxeter polytopes are the Coxeter garlands in $\H^4$ built by the compact truncated Coxeter simplex described by the Coxeter graph which is on the left side of Figure \ref{two-trun} (cf. \cite[Fig.1]{ZZ}).
This polytope is constructed by the extended Coxeter simplex with two ultraideal vertices described by the Coxeter graph on the right side of Figure \ref{two-trun}.
\begin{figure}[h]
\begin{center}
 \includegraphics [width=250pt, clip]{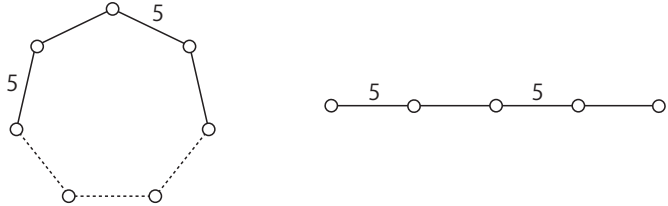}
 \caption{The Coxeter graphs of the compact truncated Coxeter $4$-simplex and its underlying extended $4$-simplex }
\label{two-trun}
\end{center}
\end{figure}
Inspired by their work, as the main result of this paper, we realize infinitely many $2$-Salem numbers as the growth rates of hyperbolic Coxeter groups in $\H^4$ as follows.

We focus on the compact Coxeter polytope $T\subset \H^4$ whose Coxeter graph is on the left side of Figure \ref{2graphs1}, and which was first described by Schlettwein \cite{Sc}.
The nodes of the graph describe the facets of the Coxeter polytope, and the (non-dotted) edges describe the dihedral angle formed by two intersecting facets (see \S \ref{hyp}). 
\begin{figure}[h]
\begin{center}
 \includegraphics [width=250pt, clip]{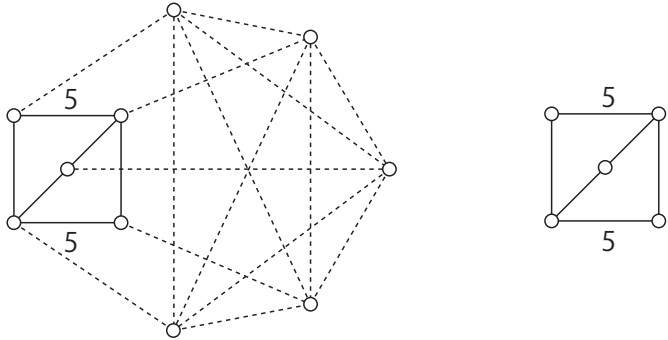}
 \caption{The Coxeter graphs of the compact totally truncated Coxeter $4$-simplex $T$ and its underlying extended $4$-simplex }
\label{2graphs1}
\end{center}
\end{figure}
This polytope is constructed by the extended Coxeter simplex of dimension $4$ with all vertices outside of $\H^4$ whose Coxeter graph is on the right side of Figure \ref{2graphs1}.
More precisely, we truncate all vertices of this simplex and replace them by facets orthogonal to all facets intersecting them, which we call orthogonal facets (see \S 4).
This construction yields the compact Coxeter polytope $T$.
This kind of construction is explained by Vinberg \cite[Proposition 4.4]{V}, for example.
Since we can glue many copies of this along isometric orthogonal facets, $T$ is the building block for infinitely many compact Coxeter polytopes which we call {\em Coxeter dominoes} (see Figure \ref{domino1}). 
Note that each orthogonal facet of $T$ is one of the three types $A$, $B$, and $C$ (see \S 5.1).

 \begin{figure}[h]
\begin{center}
 \includegraphics [width=280pt, clip]{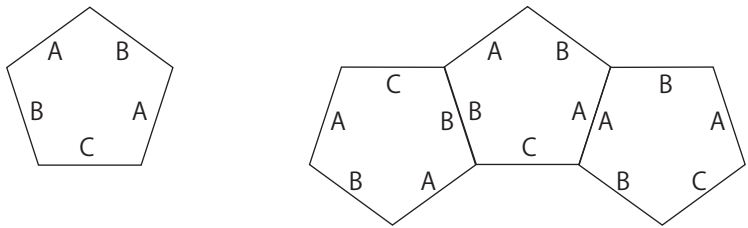}
\caption{The images of a Coxeter domino $T$ and Coxeter donimoes}
\label{domino1}
\end{center}
\end{figure}
Then we prove all that cocompact hyperbolic Coxeter groups with respect to Coxeter dominoes built by $T\subset \H^4$ have growth functions whose denominator polynomials $Q_{\ell,m,n}(t)$ satisfy the following property: all the roots of $Q_{\ell,m,n}(t)$ are on the unit circle except two pairs of real roots (see Theorem \ref{poles}).
Finally, we prove our main result providing infinite families whose growth rates are $2$-Salem numbers under certain restrictions for $(\ell, m, n)$.  



For the calculation of the growth functions of Coxeter dominoes, we adapt the formula developed by 
T. Zehrt and C. Zehrt \cite{ZZ} (see Proposition \ref{z-form-pro}), which allows us to compute the growth functions of such infinite series at once.
For the study of the denominator polynomials, we use the adaption of Kempner's result \cite{Kem} to palindromic polynomials due to T. Zehrt and C. Zehrt  about the distribution of the roots of the polynomials.
To show the irreducibility of denominator polynomials, we analyzed case by case by elementary arithmetic (cf. Theorem \ref{main}).

The paper is organized as follows.
In \S 2, we review hyperbolic space and hyperbolic Coxeter groups.
In \S 3.1, we collect some basic results of the growth functions and growth rates of Coxeter groups, and in \S 3.2, we introduce $2$-Salem numbers and their properties.
In \S 4, we describe the compact totally truncated Coxeter simplex $T$ in $\H ^4$ giving rise to  Coxeter dominoes. 
Finally, in \S 5 we state and prove the main theorem.
We also get the result that there are infinitely many $2$-Salem numbers as the growth rates of cocompact hyperbolic Coxeter groups with respect to the Coxeter garlands built by the truncated Coxeter simplex described by the Coxeter graph which is on the left side of Figure \ref{two-trun} by using the same idea as for the main theorem, but we omit to write it in this paper. 

\section*{acknowledgement}
This work is a part of my PhD thesis project supervised by Prof. Ruth Kellerhals and Prof. Yohei Komori. 
I would like to thank them for their sensible advices and instructive discussions.
This work was partially supported by the JSPS Institutional Program for Young Researcher Overseas Visits ``Promoting international young researchers in mathematics and mathematical sciences led by OCAMI''. 

\section{Cocompact hyperbolic Coxeter groups}
\label{hyp}
\subsection{Hyperbolic convex polytopes}
Let $\R^{n,1}$ be the real vector space $\R^{n+1}$ equipped with the {\em Lorentzian inner product} 
$x \circ y:=x_1y_1+\cdots +x_{n}y_{n}-x_{n+1}y_{n+1}$, where $x=(x_1,...,x_{n+1})$, $y=(y_1,...,y_{n+1}) \in \mathbb{R}^{n+1}$.
A vector $x \in \R^{n,1}$ is {\em space-like} (resp. {\em light-like, time-like}) if $x\circ x >0$ (resp. $x\circ x =0$, $x\circ x<0$). 
The set 
$$
C:=\{ x\in \R^{n,1} \,|\,  x\circ x =0\}
$$
is a cone in $\R^{n+1}$ formed by all light-like vectors.
Space-like vectors lie outside $C$ and time-like vectors lie inside $C$. 
The set
$$
\H^n:=\{x=(x_1,...,x_{n+1})\in \R^{n,1} \,|\, x\circ x =-1, x_{n+1}>0\}
$$
is a hyperboloid lying inside the upper half of $C$ and is called {\em hyperbolic $n$-space} when equipped with the metric 
$$
\cosh d_{\H}(x,y):=-x\circ y.
$$
A nonempty subset $H\subset \H^n$ is a {\em hyperplane of $\H^n$} if there exists an $n$-dimensional vector subspace $V\subset \R^{n+1}$ such that $H=V\cap \H^n$.
It is equivalent to say that $V$ is a {\em time-like subspace} of $\R^{n,1}$, that is, $V$ contains a time-like vector.
So we can represent $V$ by the {\em Lorentzian orthogonal complement} $e^L:=\{ x \in \R^{n,1}\,|\, x\circ e =0 \}$ for some unit space-like vector $e\in \R^{n,1}$ (see \cite[Exercise 3.1, No.10]{R}).
Hence, $H=e^L\cap \H^n$ and we often use the notion $\widehat{H}:=e^L$.
In a similar way, we denote by $(e^L)^-:=\{ x\in \R^{n,1} \,|\, x\circ e \leq 0\}$ the half space of $\R^{n,1}$ bounded by $\widehat{H}$, and put $H^-:=(e^L)^-\cap \H^n$ and $\widehat{H}^-:=(e^L)^-$. 
  
A {\em convex polytope} $P \subset \H^n$ of dimension $n$ is defined by the intersection of finitely many closed half spaces of $\H^n$ 
\begin{equation}
P=\bigcap _{i\in I} H_i^- 
\label{polytope}
\end{equation}
containing a nonempty open subset of $\H^n$.
We remark that $H_i=e_i^L\cap \H^n$ for a unit space-like vector $e_i \in \R^{n,1}$ such that $e_i$ is directed outwards with respect to $P$.
We always assume that for any proper subset $J\subsetneq I$, $P\subsetneq \cap _{j \in J}H_j^-$ is satisfied.
For a subset $J\subset I$, the nonempty intersection $F:=P\cap (\cap _{j\in J} H_j) $ is called a {\em face} of $P$.
Observe that $F$ itself can be considered as a convex hyperbolic polytope of some dimension.
A face $F\subset P$ of dimension $n-1$, $1$ or $0$ is called a {\em facet, edge}, or {\em vertex} of $P$ respectively. 

The mutual disposition of hyperplanes bounding $P$ is as follows (cf. \cite[p.67-p.71]{R}, \cite[p.41]{V}).
Two hyperplanes $H_i$ and $H_j$ in $\H^n$ intersect each other if and only if $|e_i\circ e_j|<1$, and the angle between $H_i$ and $H_j$ of $P$ is defined as a real number $\theta _{ij} \in [0,\pi[$ satisfying 
\begin{equation}
\cos \theta_{ij} =-e_i \circ e_j.
\label{cos}
\end{equation}
We call $\theta_{ij}$ the {\em dihedral angle} between $H_i$ and $H_j$ of $P$, and denote
$$
\angle H_i^-H_j^- :=\theta _{ij}.
$$
The hyperplanes $H_i$ and $H_j$ don't intersect if and only if 
\begin{equation*}
e_i \circ e_j\leq -1.
\label{-1}
\end{equation*}
As a remark, $e_i \circ e_j\not\geq 1$ since $H_i^- \not \subset H_j^-$ for any $i\not= j$ under the above assumption.
More precisely, if $\widehat{H_i}$ and $\widehat{H_j}$ don't intersect in $\H^n \cup C \setminus \{0\}$, then $\cosh d_{\H}(H_i, H_j)=-e_i\circ e_j$, where $d_{\H}(H_i,H_j)$ is the hyperbolic length of the unique geodesic orthogonal to $H_i$ and $H_j$.

The {\em Gram matrix  $G(P)$ of $P$} is defined as the symmetric matrix $(g_{ij}):=(e_i\circ e_j)$, $i,j\in I$, with $g_{ii}=1$.

A convex polytope $P \subset \H^n$ of dimension $n$ is called {\em acute-angled} if all of its dihedral angles are less than or equal to $\pi/2$.
For acute-angled polytopes in $\H^n$, Vinberg \cite{V} developed a complete combinatorial description in terms of its Gram matrix $G(P)$.
We will discuss it in \S 4.1.
\subsection{Hyperbolic Coxeter polytopes and their associated hyperbolic Coxeter groups}
A {\em Coxeter polytope} $P\subset \H^n$ of dimension $n$ is a convex polytope of dimension $n$ all of whose dihedral angles are of the form $\pi/m$, where $m\geq 2$ is an integer.
For the Coxeter polytope described by (\ref{polytope}), 
the group $W$ generated by the set $S=\{s_i \,|\, i\in I\}$ of reflections with respect to the bounding hyperplanes $H_i$ in $\H^n$ is a discrete subgroup of the isometry group $Isom(\H^n)$ of $\H^n$, and $P$ is a fundamental polytope of $W$. 
Furthermore, $(W,S)$ is a Coxeter system (see \S 3.1) with relations $s_i^2=id$, and $(s_is_j)^{m_{ij}}=id$ for $i\not=j$, if $\angle H_i^- H_j^-=\pi/m_{ij}$.
If $P$ is compact, we call $W$ a {\em cocompact hyperbolic Coxeter group}.  
It is convenient to associate a Coxeter graph $\Gamma$ to a Coxeter polytope $P$.
Represent each bounding hyperplane $H_i$ (or reflection $s_i\in S$) by a node $\nu _i$, and join two nodes $\nu _i$ and $\nu _j$ by a single edge labeled $m_{ij}$ if $\angle H_i^- H_j^-=\pi/m_{ij}$, $m_{ij}\geq 3$, or by a dotted edge if $m_{ij}$ is not defined. 
We do not join $\nu _i$ and $\nu _j$ if $H_i$ and $H_j$ are orthogonal, and we omit the label for $m_{ij}=3$.  

\begin{exa}
\label{237}
The following Coxeter graph describes a compact hyperbolic Coxeter triangle in $\H^2$ with angles $\pi/2, \pi/3,$ and $\pi/7$.
\begin{figure}[h]
 \includegraphics [width=60pt, clip]{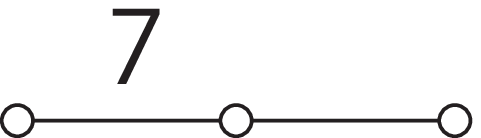}
\label{2graphs}
\end{figure}
\end{exa}

\begin{exa}
\label{zehrt}
The following Coxeter graph represents a compact hyperbolic Coxeter polytope in $\H^4$ $($cf. \cite[Example 2]{ZZ}$)$.
\begin{figure}[h]
 \includegraphics [width=150pt, clip]{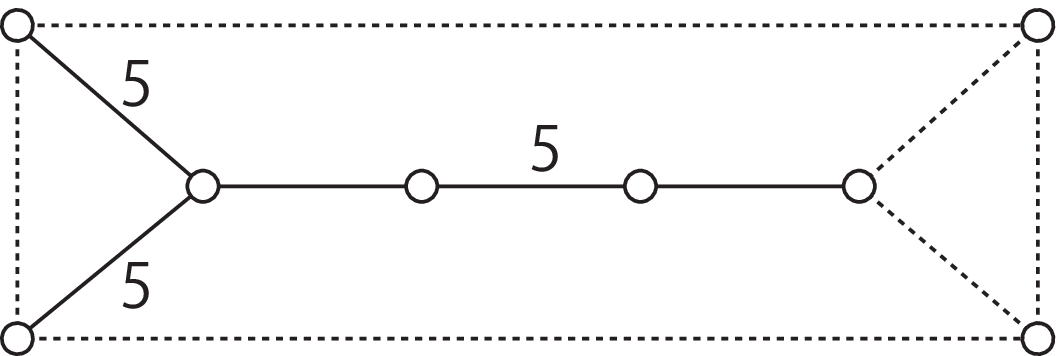}
\label{2graphs}
\end{figure}
\end{exa}

\section{Growth rates of cocompact hyperbolic Coxeter groups}
\subsection{Growth functions of Coxeter groups}
\label{growth}
In this section, we shall introduce the quantity $\tau$ for the hyperbolic Coxeter groups.
At first, we shall consider the general situation.
Let $(W, S)$ be a pair consisting of a group $W$ and its finite generating set $S$.
The {\em word length} of an element $w \in W$ with respect to $S$ is defined by 
$$
l_S(w):=\min \, \{ k\in \N \cup \{0\}\,|\, w=s_1s_2\cdots s_k, s_i\in S \}.
$$
By convention, $l_S(id)=0$. 
The {\em growth series} of $(W,S)$ is defined by the power series
\begin{equation}
f_S(t):=\sum _{w\in W} t^{l_S(w)} = \sum _{k\geq 0} a_k t^k = 1+|S|t+\cdots,\, t\in \C,
\label{series}
\end{equation}
where $a_k$ is the number of elements $w \in W$ with $l_S(w)=k$ and $|S|$ denotes the number of $S$.
The {\em growth rate} of $(W,S)$ is defined by
\begin{equation}
\tau:= \limsup _{k\rightarrow \infty} \sqrt[k]{a_k},
\label{rate}
\end{equation}
 which is the inverse of the radius of convergence $R$ of $f_S(t)$, that is, $\tau =R^{-1}$.

From now on, we focus on the growth series and the growth rate of a Coxeter group, wherefore we review the notions of Coxeter groups.
A pair $(W, S)$ is called a {\em Coxeter system} if generators $s,t \in S$ satisfy relations of the type $(st)^{m_{s,t}}=id$, $m_{s,t}=m_{t,s}$, and $m_{s,s}=1$.
We call the group $W$ itself a {\em Coxeter group}.
In the sequel, we often do not distinguish between Coxeter group and its underlying Coxeter system.
It is convenient to use the {\em Coxeter graph} $\Gamma$ associated to $(W, S)$ whose nodes $\nu _s$ correspond to the generators $s\in S$, and two nodes $\nu_s, \nu_t$ are joined as follows.
They are joined by a single edge labeled $m_{s,t}$ if $m_{s,t}\geq 3$ and are not joined if $m_{s,t}=2$.
Usually, the label is omitted if $m_{s,t}=3$.
Two nodes $\nu _s$ and $\nu _t$ are joined by a dotted edge if $m_{s,t}$ is not defined.
All connected Coxeter graphs with finite associated Coxeter groups are classified \cite{Co}.

First, we consider a finite Coxeter group $(W,S)$.
It is obvious that its growth series (\ref{series}) is a polynomial. 
Solomon \cite[Corollary 2.3]{So} gave an explicit formula for the growth polynomial of $(W,S)$, which allows us to compute it by using the {\em exponents} of $(W,S)$, that is, 
\begin{equation}
f_S(t)=\prod _{i=1}^k [m_i+1],
\label{solo}
\end{equation}
where $[m]=1+t+\cdots +t^{m-1}$, $1=m_1\leq m_2 \leq ...\leq m_k$.

Next, we consider an infinite Coxeter group $(W,S)$.
Steinberg \cite[Corollary 1.29]{St} derived a formula for the growth series $f_S(t)$ of $(W,S)$, which allows us to compute it by using the growth polynomial $f_T(t)$ of its finite Coxeter subgroups $(W_T,T)$ generated by a subset $T\subset S$, that is, 
\begin{equation}
\frac{1}{f_S(t^{-1})}=\sum _{\stackrel{T \subset S}{|W_T|<\infty}} \frac{(-1)^{|T|}}{f_T(t)}=1-\frac{|S|}{[2]}+\cdots.
\label{st}
\end{equation}
Here, each $f_T(t)$ is of the form (\ref{solo}).
This means that $\frac{1}{f_S(t^{-1})}$ is expressible as a rational function, say
\begin{equation*}
\frac{1}{f_S(t^{-1})}=\frac{\widetilde{q}(t)}{\widetilde{p}(t)},
\label{prime}
\end{equation*}
where $\widetilde{q}(t),\widetilde{p}(t) \in \Z[t]$ are relatively prime and monic of the same degree, say $n$.
Furthermore, the constant term of $\widetilde{p}(t)$ is $\pm 1$ by (\ref{solo}) and (\ref{st}).
Hence, we have
\begin{equation}
f_S(t)=\frac{p(t)}{q(t)},
\label{rational}
\end{equation}
where $p(t)=t^n\widetilde{p}(t^{-1}), q(t)=t^n\widetilde{q}(t^{-1}) \in \Z[t]$.
Both polynomials $p(t)$ and $q(t)$ are relatively prime over $\Z$, and we call the growth series $f_S(t)$ described by the form (\ref{rational}) the {\em growth function} of $(W, S)$. 
Observe that the smallest positive root of $q(t)$ equals the radius of convergence $R$.
This means $q(R)=R^n\widetilde{q}(R^{-1})=0$, and since $\widetilde{q}(t) \in \Z[t]$ is monic, the growth rate $\tau =R^{-1}$ is an algebraic integer. 

\begin{exa}
The growth function of the cocompact hyperbolic Coxeter group described by the Coxeter graph in Example \ref{237} is as follows: 
\begin{equation}
f_S(t)=\frac{(t+1)^2 (t^2+t+1) (t^6+t^5+t^4+t^3+t^2+t+1)}{t^{10}+t^9-t^7-t^6-t^5-t^4-t^3+t+1}.
\label{237f}
\end{equation}
\end{exa}

\begin{exa}
The growth function of the cocompact hyperbolic Coxeter group described by the Coxeter graph in Example \ref{zehrt} is as follows  $($cf. \cite[Theorem 1]{ZZ}$) :$
\begin{equation}
f_S(t)=\frac{(t+1)^4 (t^2-t+1) (t^2+t+1) (t^4-t^3+t^2-t+1) (t^4+t^3+t^2+t+1)}{t^{16}-4 t^{15}+t^{14}+t^{12}+t^{11}+2 t^9+2 t^7+t^5+t^4+t^2-4 t+1}.
\label{zehrtf}
\end{equation}
\end{exa}

\subsection{Growth rates and $2$-Salem numbers}
In the previous subsection 3.1, we see that the growth rate of a Coxeter group is an algebraic integer.
Cannon, Wagreich and Parry (see \cite{Pa}, for example) showed that the growth rate of a cocompact hyperbolic Coxeter group in $\H^2$ or $\H^3$ is a Salem number or a quadratic unit.
We recall the notion of Salem numbers.

A {\em Salem number} is a real algebraic integer $\alpha >1$ all of whose other conjugate roots $\omega $  satisfy $|\omega |\leq 1$ and at least one conjugate root is on the unit circle  (cf. \cite[Definition 5.2.2]{BDGPS}, \cite[p.293]{GhHi}). 
Call the minimal polynomial $p_{\alpha}(t)$
of $\alpha$ a  {\em Salem polynomial}.
The following proposition is a well known fact about Salem polynomials (cf. \cite[p.294]{GhHi}). 

\begin{prop}
A Salem polynomial $p_{\alpha}(t)$ is palindromic of even degree, that is, $p_{\alpha}(t)=t^np_{\alpha}(t^{-1})$, where $n$ is the degree of $p_{\alpha}(t)$.
\end{prop}

\begin{proof}
Since $p_{\alpha}(t)\in \Z[t]$ has a root $\omega_0$ on the unit circle, then $\overline{\omega_0}=\omega_0^{-1}$ is also a root of $p_{\alpha}(t)$.
Hence, $t^np_{\alpha}(t^{-1})$ is also the minimal polynomial of $\omega_0$ and then there exists a constant $c\in \Z$ such that 
$$
p_{\alpha}(t)=ct^np_{\alpha}(t^{-1}).
$$ 
As a consequence, $\alpha^{-1}$ is also a root of $p_{\alpha}(t)$, so that any root $\omega$ of $p_{\alpha}(t)$ except $\alpha$ and $\alpha^{-1}$ is on the unit circle, where $\omega^{-1}$ is also a root of $p_{\alpha}(t)$.
Therefore the constant term of $p_{\alpha}(t)$ should be $1$, which implies that $c=1$, that is,  $p_{\alpha}(t)=t^np_{\alpha}(t^{-1})$.
We conclude that $p_{\alpha}(t)$ is a palindromic polynomial of even degree.
\end{proof}


\begin{exa} 
\label{Lehmer-p}
The growth rate of the cocompact hyperbolic Coxeter group described by the Coxeter graph in Example \ref{237} is a Salem number.
More precisely, the denominator polynomial 
\begin{equation*}
L(t)= t^{10}+t^9-t^7-t^6-t^5-t^4-t^3+t+1
\label{lehmer}
\end{equation*}
of  its growth function (\ref{237f}) is a Salem polynomial, and its positive root $\alpha_L\approx 1.17628$, which is a Salem number, is the growth rate. 
Note that $L(t)$ and $\alpha_L$ are known as the Lehmer polynomial and its Lehmer's number, respectively $($cf. \cite{L}$)$.
Figure \ref{Lehmer} describes the distribution of the roots of $L(t)$.
\begin{figure}[h]
\begin{center}
 \includegraphics [width=115pt, clip]{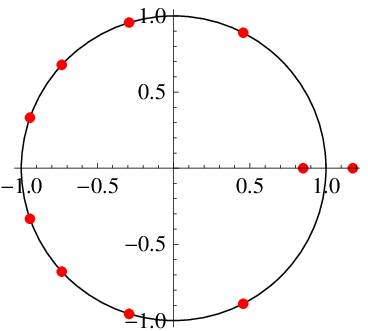}
\caption{}
\label{Lehmer}
\end{center}
\end{figure}
\end{exa}

 
As we see, the minimal polynomial of a Salem number has roots on the unit circle except for a  reciprocal pair of positive real roots. 
Focussing on the numbers of roots outside the unit circle and the existence of a root on the unit circle, a $j$-Salem number has been defined by Kerada \cite[Definition 2.1]{Ker}, which is a generalization of a Salem number. 
In this paper, we focus on a $2$-Salem number. 
Note that Samet \cite{Sam} also defined a set of algebraic integers corresponding to $2$-Salem numbers.
 \begin{defi} 
 \label{2Salem}
A $2$-Salem number is an algebraic integer $\alpha $ such that $|\alpha |>1$ and $\alpha$ has one conjugate root $\beta$ different from $\alpha$ satisfying $|\beta|>1$ while other conjugate roots $\omega$ satisfy $|\omega|\leq 1$ and at least one of them is on the unit circle.
Call the minimal polynomial $p_{\alpha}(t)$ of $\alpha$ a $2$-Salem polynomial.
\end{defi}

As in the case of a Salem polynomial, a $2$-Salem polynomial $p_{\alpha}(t)$ becomes a palindromic polynomial of even degree.
As a consequence, $\alpha^{-1}$ and $\beta^{-1}$ are roots of $p_{\alpha}(t)$, and all roots different from $\alpha,\alpha ^{-1}, \beta , \beta ^{-1}$ lie on the unit circle and are complex (cf. Figure \ref{2-Salem}).
So a $2$-Salem polynomial has even degree $n\geq 6$, and note that $\beta$ becomes also a $2$-Salem number.

In \cite{ZZ}, T. Zehrt and C. Zehrt found infinitely many cocompact Coxeter groups in $\H^4$ whose denominators $q(t)$ of the growth functions $f_S(t)$ have the following property: all the roots of the polynomial $q(t)$ are on the unit circle except exactly two pairs of real roots.    
We notice that if $q(t)$ is irreducible, it becomes a $2$-Salem polynomial. 
This motivated us to investigate whether $2$-Salem numbers appear as growth rates of such groups.
As a first observation, we get the following proposition.

\begin{prop}
The growth rate of the cocompact hyperbolic Coxeter group described by the Coxeter graph in Example \ref{zehrt}  is a $2$-Salem number.
\end{prop}

\begin{proof}
We prove that the denominator polynomial 
\begin{equation*}
D(t)=t^{16}-4 t^{15}+t^{14}+t^{12}+t^{11}+2 t^9+2 t^7+t^5+t^4+t^2-4 t+1
\end{equation*}
of the growth function (\ref{zehrtf}) is a $2$-Salem polynomial.
First,  by \cite[Theorem 2]{ZZ}, $D(t)$ has six pairs of complex roots on the unit circle and two pairs of positive roots $\alpha, \alpha ^{-1}$ and $\beta, \beta ^{-1}$, where $\alpha, \beta >1$ (see Figure \ref{2-Salem}). 

The irreducibility of $D(t)$ in $\Z[t]$ is proved by Cohn's theorem (see \cite[Theorem 1]{M}) as follows.
For a polynomial $f(t)=t^m+a_{m-1}t^{m-1}+\cdots +a_1t+a_0$ of degree $m$ in $\mathbb{Z}[t]$, set 
\begin{equation*}
H=\max_{0\leq i\leq m-1} |a_i|.
\label{H}
\end{equation*}
If $f(n)$ is prime for some integer $n\geq H+2$, then $f(t)$ is irreducible in $\Z[t]$ by Cohn's criterion.
In fact, $H=4$ for $D(t)$,  and $D(186)=20080678392852674723847588201\\73242349$ is a prime number.
So $D(t)$ is irreducible.
This completes the proof.
Furthermore, two positive roots $\alpha \approx 3.70422$ and $\beta \approx1.24202$ are $2$-Salem numbers.
\begin{figure}[h]
\begin{center}
 \includegraphics [width=230pt, clip]{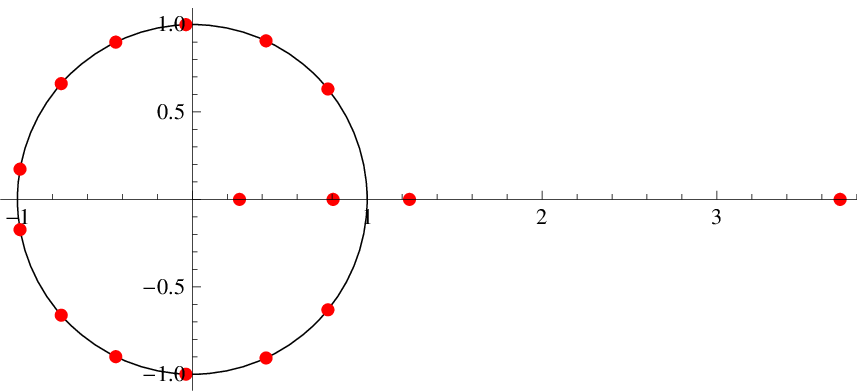}
\caption{}
\label{2-Salem}
\end{center}
\end{figure}
\end{proof}

In the next chapter, we construct families of infinitely many cocompact hyperbolic Coxeter groups in $\H^4$, which are different from those constructed by T. Zehrt and C. Zehrt, and prove that their growth rates are $2$-Salem numbers. 

\newpage
\section{Construction of a Coxeter domino $T$}
T. Zehrt and C. Zehrt \cite[Theorem 2]{ZZ} described the characteristic distribution of roots of the denominator polynomials of the growth functions of Coxeter garlands in $\H^4$ built by a particular compact truncated Coxeter $4$-simplex with two orthogonal facets.
In this paper, we have similar, but more detailed results for the denominator polynomials of the growth functions of Coxeter dominoes constructed by the Coxeter domino $T\subset \H^4$, which leads us to the connection with $2$-Salem numbers. 

To begin with, we are interested in the Coxeter system $(W,S)$ having the following Coxeter graph  $\Gamma$ (see Figure \ref{sc-graph}).
It was first described by Schlettwein in \cite{Sc} (unpublished).

\begin{figure}[h]
\begin{center}
 \includegraphics [width=50pt, clip]{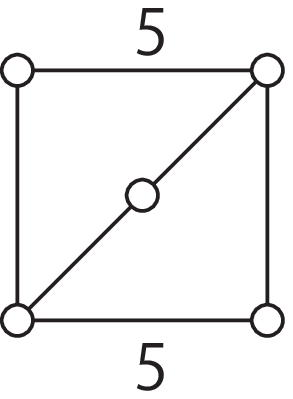}
\caption{}
\label{sc-graph}
\end{center}
\end{figure}
At first, notice that this graph describes an infinite Coxeter group since it does not belong to the well-known list of connected Coxeter graphs of finite Coxeter groups.
In \S4.1, we will explain in detail how to realize $(W,S)$ as a geometric Coxeter group in $\H^4$.

\subsection{The Coxeter simplex $P\subset \H^4$ with ultraideal vertices}
We consider the matrix $G$ related to the Coxeter graph $\Gamma$ in Figure \ref{sc-graph}, which is defined by $(g_{ij})=(-\cos\frac{\pi}{m_{ij}})$:
\begin{equation}
G=
\begin{pmatrix}
1&-\cos \frac{\pi}{5}&0&-\frac{1}{2}&0\\
-\cos \frac{\pi}{5}&1&-\frac{1}{2}&0&-\frac{1}{2}\\
0&-\frac{1}{2}&1&-\cos \frac{\pi}{5}&0\\
-\frac{1}{2}&0&-\cos \frac{\pi}{5}&1&-\frac{1}{2}\\
0&-\frac{1}{2}&0&-\frac{1}{2}&1
\end{pmatrix}.
\label{mat}
\end{equation}
The signature of $G$ equals $(4,1)$.
In fact, $\det G<0$ and the principal submatrix formed by the first three rows and columns is positive definite (and related to the Cartan matrix of  $H_3$). 
Therefore, $G$ has precisely one negative and four positive eigenvalues (see \cite{Sa}, for example).
By a result of Vinberg \cite[Theorem 2.1]{V}, an indecomposable symmetric real matrix $G=(g_{ij})$, with $g_{ii}=1$ and $g_{ij}\leq 0$ for $i\not=j$, is the Gram matrix of an acute-angled polytope of dimension $n$ in $\H^n$ (up to isometry) if the signature of $G$ equals $(n,1)$.
Therefore, $G$ is the Gram matrix of a Coxeter polytope $P\subset \H^4$ given by the Coxeter graph $\Gamma $.
However, $P$ is not compact since it is not in the list of Coxeter graphs of compact Coxeter simplices in $\H^4$ classified by Lann\'er \cite{La} (cf. \cite[p.141]{Hu}). 
In the sequel, we associate to $P$ a new compact polytope $T\subset \H^4$ by truncation.

Let $P=\cap _{i\in I} H_i^- \subset \H^4$ be the Coxeter polytope having a Coxeter graph $\Gamma$ as in Figure \ref{sc-graph}, where $I=\{1,...,5\}$.
Represent $H_i=e_i^L\cap \H^4$ as usually such that $G=(e_i\circ e_j)_{i,j\in I}$.
It follows that $P$ is an extended $4$-simplex in $\R^{4,1}$, bounded by five hyperplanes, and of infinite volume in $\H^4$.
In fact, all vertices of $P$ lie outside of $\H^4$.
To see this, let us use Vinberg's description of faces in terms of submatrices of $G$.
 
Observe that all principal submatrices of order four in $G$ are of signature $(3,1)$ and described by the Coxeter graphs in Figure \ref{graphABC} (cf. \cite[p.53]{La}).
Furthermore, all other principal submatrices of $G$, that is, those of order less than four, are positive definite and described by positive definite graphs, that is, by certain finite Coxeter groups. 
Then, by a result of Vinberg \cite[Theorem 3.1]{V}, all the faces of $P$ of positive dimension are hyperbolic polytopes while the vertices of $P$ do not belong to $\H^4$.
In fact, $\cap _{j\in I_i}H_j=\emptyset$, for $I_i:=I\setminus \{i\}$, $i=1,...,5$.
However,  the set $\cap _{j\in I_i} \widehat{H_j}$ is a one-dimensional subspace in $\R^{4,1}$ since $e_i, i\in I$, are linearly independent.
This line can be represented by a positive space-like unit vector $v_i \in \R^{4,1}$ such that $v_i$ is a vertex of $P$ outside $\H^4$ (see also (\ref{vertex})).
We call $v_i$ an {\em ultraideal vertex} of $P$.

\subsection{Construct the compact totally truncated Coxeter simplex $T \subset \H^4$ from $P$}
Now we will truncate all ultraideal vertices off from $P$ in order to obtain a compact Coxeter polytope $T$ as follows. 
The set $\cap _{j\in I_i}\widehat{H}^-_j$ is a $4$-dimensional simplicial cone in $\R^{4,1}$, whose apex $v_i$ is  outside of $\H^4$.
Moreover, the hyperplane $H_{v_i}:=v_i^L\cap \H^4$ intersects all of the hyperplanes $H_j$, $j\in I_i$, orthogonally, that is, 
\begin{equation}
\angle H_{v_i}^- H_j^-=\frac{\pi}{2},\, j\in I_i,
\label{ortho}
\end{equation}
 since $v_i\circ e_j=0$ for all $j\in I_i$ by definition of $v_i$, and by (\ref{cos}).
In this situation, by a result of Vinberg \cite[Proposition 4.4]{V}, $P\cap H_{v_i}^-$ is also a convex polytope in $\H^4$.
By performing this operation for each ultraideal vertex $v_i$, $i=1,...,5$, we get a new polytope
$$
T:=P\cap (\cap _{i\in I} H_{v_i}^-) \subset \H^4,
$$  
which we shall call {\em totally (orthogonally) truncated simplex} (see Figure \ref{bfaf-truncate}).
A facet $P\cap H_{v_i}$ is called an {\em orthogonal facet} of $T$.
\begin{figure}[h]
 \includegraphics [width=250pt, clip]{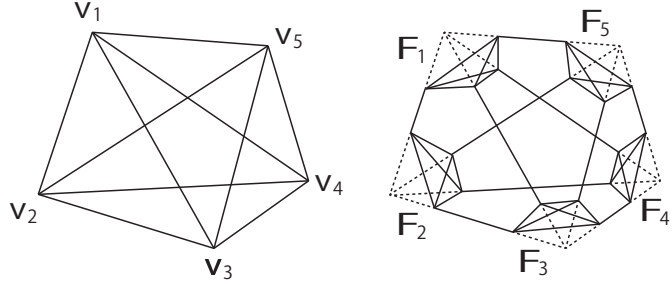}
\caption{The passage to the compact totally truncated Coxeter 4-simplex $T$ where each ultraideal vertex $v_i$ is replaced by the orthogonal facet $F_i$}
\label{bfaf-truncate}
\end{figure}

Let us describe the ultraideal vertices $v_1,...,v_5$ of $P$ in the following explicit way (cf. \cite{Lu}).
Let $cof_{ij}(G):=(-1)^{i+j}\det G_{ij}$ be the cofactor of $G$, where $G_{ij}$ is the submatrix obtained by removing the $i$-th row and $j$-th column of $G$.
Since the $ij$-th coefficient of $G^{-1}$ equals $\frac{1}{\det G}\, cof_{ji}(G) $, then by comparing the $ij$-th coefficient of $G^{-1}G=I$, 
$$
\frac{1}{\det G}\, \sum_{k=1}^{5} cof_{ik}(G) g_{kj}=\delta_{ij}
$$
and then
\begin{equation*}
\left(\sum_{k=1}^{5} cof_{ik}(G) e_k \right) \circ e_j=\delta_{ij}\det G.
\end{equation*}
If we set
\begin{equation*}
w_i:=\sum _{k=1}^{5} cof_{ik}(G)e_k,\, 1\leq i\leq 5,
\end{equation*}
then 
\begin{equation*}
w_i \circ w_j=cof_{ij}(G)\det G 
\begin{cases}
<0 \;\; \mbox{if $i\not=j$}\\
>0 \;\; \mbox{if $i=j$}.
\end{cases}
\end{equation*}
Hence, $\{w_1,...,w_5\}$ is a linearly independent set. 
Now, consider
\begin{equation}
v_i=\frac{w_i}{w_i \circ w_i}=\frac{w_i}{\sqrt{cof_{ii}(G)\det G}}.
\label{vertex}
\end{equation}
Then, for each $i$, 
\begin{equation}
\begin{split}
&v_i \circ v_i =1,\\
&v_i \circ e_j =0\; \mbox{for $j\not=i$},\\
&v_i \circ e_i =-\sqrt{\frac{\det G}{cof_{ii}(G)}}<0.
\end{split}
\label{value}
\end{equation}
So, for each $i$, $\{ v_j \,|\, j\in I_i\}$ spans $\widehat{H_i}$ since it is a linearly independent set, and $v_i$ spans $\cap _{j\in I_i} \widehat{H_j}$.
Hence $v_i$ is indeed an ultraideal vertex of $P$. 

By direct calculation, based on (\ref{mat}) and (\ref{value}), we get
$
v_i\circ e_i <-1
$
for each $i\in I$, which means that the orthogonal facet $H_{v_i}$ and the facet $H_i$ opposite to it are disjoint in $T$ \cite[\S 3]{R}:
\begin{equation}
H_{v_i}\cap H_i =\emptyset.
\label{i}
\end{equation}
Similarly,
$
v_i \circ v_j
=\frac{-cof_{ij}(G)}{\sqrt{cof_{ii}(G)cof_{jj}(G)}}
<-1
$
for each $j\in I_i$, which means that the orthogonal facets are mutually disjoint:
\begin{equation}
H_{v_i}\cap H_{v_j} =\emptyset  \;\mbox{if $i\not=j$}.
\label{ij}
\end{equation}
By combining (\ref{ortho}), (\ref{i}), and (\ref{ij}), it follows that $T$ is described by the Coxeter graph $\Gamma ^*$ in Figure \ref{5trun-4simp1} (see below).
In particular, $T$ is a Coxeter polytope.
\begin{figure}[h]
\begin{center}
 \includegraphics [width=150pt, clip]{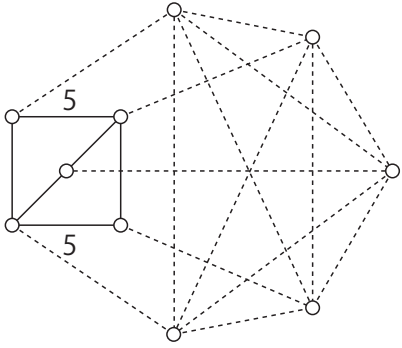}
\caption{The Coxeter graph $\Gamma^*$ of the compact totally truncated Coxeter 4-simplex $T$}
\label{5trun-4simp1}
\end{center}
\end{figure}

Let $I^*=\{1,...,10\}$ be the indexed set of ten facets of $T$.
For the compactness of $T$, observe first that all order four subgraphs $\Gamma(4)$ of $\Gamma$ are of signature $(3,1)$, and each subgraph $\Gamma(2)$ of type  \includegraphics [width=30pt, clip]{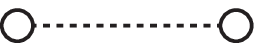} in $\Gamma ^*$ is of signature $(1,1)$. 
Next, let $J\subset I^*$ be a subset corresponding to one of the subgraphs $\Gamma (j), j=2,4$, above.
Consider $J'\subset I^*$ indexing all nodes in $\Gamma^*$ which are {\em not} connected to nodes in $J$.
Form $N(J):=J\cup J'$. 
By applying Vinberg's criterion \cite[Theorem 4.1]{V}, $T$ is compact because it is easily checked that $\cap _{i\in N(J)}\widehat{H_i}=\{0\}$.

Since $T$ will be the building block for new Coxeter polytopes (see \S \ref{g-f-domino}), we call $T$ a {\em Coxeter domino}.

\section{Main theorems}
\subsection{Growth functions of Coxeter dominoes}
\label{g-f-domino}
As we  see, $T$ has five orthogonal facets and each of them is one of the three kinds of Coxeter $3$-simplices in $\H^3$ described by the Coxeter graphs in Figure \ref{graphABC}.
More precisely, $T$ has two orthogonal facets of type $A$ and $B$, and one orthogonal facet of type $C$.

\begin{figure}[h]
\begin{center}
 \includegraphics [width=180pt, clip]{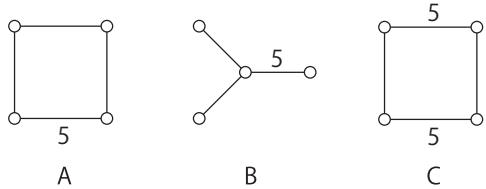}
\caption{The three Coxeter graphs of the hyperbolic Coxeter simplex of dimension $3$ which describe the figures of three types of orthogonal facets of $T$}
\label{graphABC}
\end{center}
\end{figure}

By gluing copies of $T$ along the orthogonal facets of the same type, we obtain a new polytope which is again a Coxeter polytope.
By gluing over and over again, we obtain infinitely many Coxeter polytopes in $\H^4$.  
This construction is established by Makarov \cite{Ma} and is also explained in \cite[Chapter II, \S 5]{V}.
Vinberg uses the term ``garlands" (cf. \cite[p.62]{V}) for the resulting Coxeter polytopes obtained by gluing together the truncated Coxeter simplices having two orthogonal facets (see also \cite{ZZ}).

Let us consider the growth functions of the Coxeter groups with respect to the Coxeter dominoes built by $T$.
At first, the growth function $W(t)$ of the Coxeter domino $T$ itself is calculated by using the formula (\ref{solo}) and (\ref{st}):
\begin{equation}
\begin{split}
\frac{1}{W(t^{-1})}
&=\frac{1}{W(t)}\\
&=1-\frac{10}{[2]}+\left\{ \frac{2}{[2,5]}+\frac{4}{[2,3]}+\frac{24}{[2,2]}\right\}\\
&-\left\{ \frac{6}{[2,6,10]}+\frac{3}{[2,3,4]}+\frac{6}{[2,2,5]}+\frac{12}{[2,2,3]}+\frac{13}{[2,2,2]}\right\}\\
&+\left\{ \frac{12}{[2,2,6,10]}+\frac{6}{[2,2,3,4]}+\frac{2}{[2,2,2,2]}\right\},
\end{split}
\label{W(t)}
\end{equation}
where the first equality of (\ref{W(t)}) comes from \cite {Se} (and \cite[Corollary]{CD}) since $T$ is compact.
And then
$$
W(t)=\frac{P(t)}{Q(t)},
$$
where 
\begin{equation}
\begin{split}
P(t)
&=[2,4,6,10]\\
&=\Phi _2(t)^4 \Phi _3(t) \Phi _4(t) \Phi _5(t) \Phi _6(t) \Phi _{10}(t),
\end{split}
\label{P(t)}
\end{equation}
and
\begin{equation}
\begin{split}
Q(t)&=t^{18}-6 t^{17}+3 t^{16}-5 t^{15}+5 t^{14}-t^{13}+9 t^{12}+11 t^{10}-2 t^9+11 t^8\\
&+9 t^6-t^5+5 t^4-5 t^3+3 t^2-6 t+1,
\end{split}
\label{Q(t)}
\end{equation}
which is a palindromic polynomial of degree $18$.
Here $\Phi _i(t)$ is the $i$-th cyclotomic polynomial.
To compute the growth functions for Coxeter dominoes, we use the following key formula.

\begin{prop}\cite[Corollary 2]{ZZ}
\label{z-form-pro}
Consider two Coxeter $n$-polytopes $P_1$ and $P_2$ having the same orthogonal facet of type $F$ which is a Coxeter $(n-1)$-polytope, and their growth functions $W_1(t)$, $W_2(t)$ and $F(t)$ respectively.
Then the growth function $(W_1*_FW_2)(t)$ of the Coxeter polytope obtained by gluing $P_1$ and $P_2$ along $F$ is given by
\begin{equation*}
\frac{1}{(W_1*_FW_2)(t)}=\frac{1}{W_1(t)}+\frac{1}{W_2(t)}+\frac{t-1}{t+1}\frac{1}{F(t)}.
\label{z-form}
\end{equation*}
\end{prop}

Now we have the formula for the growth functions of the Coxeter dominoes built by $T$ as follows.
\begin{cor}
Consider $n+1$ copies of $T$, and obtain from them one Coxeter polytope by $n$-times gluing. 
If this polytope is obtained by $\ell$-times gluing along the orthogonal facet of type $A$,  $m$-times along $B$ and $(n-\ell-m)$-times along $C$, where $\ell+m \leq n$ (cf. Figure \ref{domino2}), then $n-\ell-m\leq (n+1)/2$ if $n$ is odd, while $n-\ell-m\leq n/2$ if $n$ is even, and
its growth function $W_{\ell,m,n}(t)$ is given by
\begin{equation}
\frac{1}{W_{\ell,m,n}(t)}=\frac{n+1}{W(t)}+\frac{t-1}{t+1}\left(\frac{\ell}{A(t)}+\frac{m}{B(t)}+\frac{n-\ell-m}{C(t)}\right).
\label{lmn}
\end{equation}
\begin{figure}[h]
\begin{center}
 \includegraphics [width=210pt, clip]{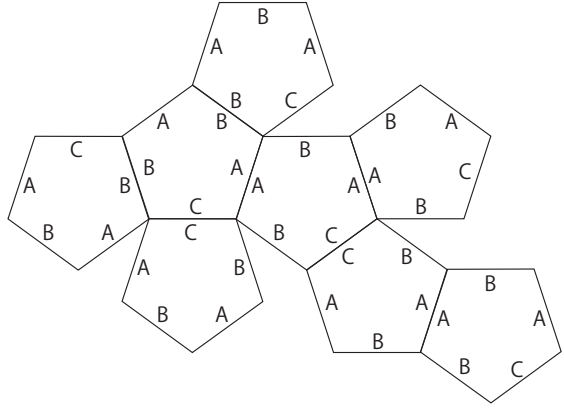}
\caption{The image of the Coxeter polytope obtained by $7$-times gluing of $T$ satisfying $(\ell, m, n)=(3,2,7)$}
\label{domino2}
\end{center}
\end{figure}
\label{garl-formu1}
\end{cor}

Note that $W_{\ell, m,n}(t)$ only depends on $(\ell,m,n)$ and does not directly depend on the resulting polytope.

\begin{theo}
The growth function $W_{\ell,m,n}(t)$ is the rational function $\displaystyle \frac{P_{\ell,m,n}(t)}{Q_{\ell,m,n}(t)}$, where
\begin{equation}
\begin{split}
P_{\ell,m,n}(t)
&=[2,4,6,10]\\
&=\Phi _2(t)^4 \Phi _3(t) \Phi _4(t) \Phi _5(t) \Phi _{6}(t) \Phi _{10}(t)
\end{split}
\label{Plmn}
\end{equation} and
\begin{equation}
\begin{split}
Q_{\ell,m,n}(t)&=t^{18}-(4n+6)t^{17}+(2n-m+3)t^{16}-(3n-m+\ell+5)t^{15}\\
&+(5n-3m+5)t^{14}-(n-4m+1)t^{13}+(8n-4m+\ell+9)t^{12}\\
&+(5m-\ell)t^{11}+(10n-5m+\ell+11)t^{10}-(2n-6m+2)t^9\\
&+(10n-5m+\ell+11)t^{8}+(5m-\ell)t^{7}+(8n-4m+\ell+9)t^{6}\\
&-(n-4m+1)t^{5}+(5n-3m+5)t^{4}-(3n-m+\ell+5)t^{3}\\
&+(2n-m+3)t^{2}-(4n+6)t+1
.
\end{split}
\label{Qlmn}
\end{equation}

Furthermore, $P_{\ell,m,n}(t)$ and $Q_{\ell,m,n}(t)$ are relatively prime.
\end{theo}

\begin{proof}
The growth function $W(t)=P(t)/Q(t)$ of $T$ is already given by (\ref{P(t)}) and (\ref{Q(t)}).
The growth functions $A(t)$, $B(t)$, and $C(t)$ of the orthogonal facets of types $A$, $B$, and $C$ in Figure \ref{graphABC} are calculated by using the formula (\ref{solo}) and (\ref{st}) as follows (cf. \cite{W});
\begin{equation*}
\begin{split}
A(t)
&=-\frac{(t+1)^3 (t^2+1) (t^2-t+1) (t^4-t^3+t^2-t+1)}{(t-1) (t^{10}-2 t^9+t^8-2 t^6+2 t^5-2 t^4+t^2-2 t+1)},\\
\end{split}
\end{equation*}
\begin{equation*}
\begin{split}
B(t)
&=-\frac{(t+1)^3 (t^2+1) (t^2-t+1) (t^4-t^3+t^2-t+1)}{(t-1) (t^{10}-2 t^9+2 t^8-2 t^7+t^6-t^5+t^4-2 t^3+2 t^2-2 t+1)},
\end{split}
\end{equation*}
\begin{equation*}
\begin{split}
C(t)
&=-\frac{(t+1)^3 (t^2-t+1) (t^4-t^3+t^2-t+1)}{(t-1) (t^2+1) (t^6-2 t^5-t^4+3 t^3-t^2-2 t+1)}.
\end{split}
\end{equation*}
Then by (\ref{lmn}) in Corollary \ref{garl-formu1}, (\ref{Plmn}) and (\ref{Qlmn}) are obtained.

It is easily seen that $Q_{\ell,m,n}(t)$ is not divided by the cyclotomic polynomials $\Phi _2(t)$,  $\Phi _3(t)$,  $\Phi _4(t)$,  $\Phi _5(t)$,  $\Phi _6(t)$, and $\Phi _{10}(t)$.
\end{proof}

\subsection{Growth rates of Coxeter dominoes and $2$-Salem numbers}
Now, we will show that there are infinitely many $2$-Salem numbers as growth rates of $W_{\ell,m,n}(t)$. 
Recall that $\ell,m,n\in \N \cup \{0\}$ satisfy the inequalities
\begin{equation}
\left\{
\begin{array}{l}
\ell+m\leq n\\ 
n-\ell-m\leq (n+1)/2  \;\;\; \mbox{if $n$ is odd}\\
n-\ell-m\leq n/2 \;\;\; \mbox{if $n$ is even} 
\end{array}
\right.
\label{ine}
\end{equation}
as in Corollary \ref{garl-formu1}.
The next proposition is an adaption of Kempner's result \cite{Kem} to a palindromic polynomial in $\Z[t]$ to investigate its number of positive real roots and roots on the unit circle. 
It is due to T. Zehrt and C. Zehrt.

\begin{prop}$($\cite[Proposition 1 and Corollary 1]{ZZ}$)$
Let $f \in \mathbb{Z}[t]$ be a palindromic polynomial of even degree $n \geq 2$ with $f(\pm1) \not = 0$ and let
$$
g(t)=(t-i)^n f\left( \frac{t+i}{t-i} \right) = (t+i)^n f\left( \frac{t-i}{t+i} \right).
$$
Then $g(t)$ is a polynomial in $\mathbb{Z}[t]$ of degree $n$ and an even function.
Furthermore, if we consider $g(t)$ as a function of $u=t^2$, then the roots of $f(t)$ and $g(u)$ are related as follows. 
\begin{enumerate}
 \item $f(t)$ has $2k$ roots on the unit circle if and only if $g(u)$ has $k$ positive real roots.
 \item $f(t)$ has $2\ell$ real roots if and only if $g(u)$ has $\ell$ negative real roots.
\end{enumerate}
\label{root1} 
\end{prop}

Applying Proposition \ref{root1}, we have the following result about the roots of $Q_{\ell, m,n}(t)$. 

\begin{theo}
\label{poles}
\begin{enumerate}
 \item $Q_{\ell,m,n}(t)$ has exactly seven pairs of complex roots on the unit circle and exactly two pairs of real roots.
 \item The two pairs of real roots $(\alpha _{\ell,m,n}, \frac{1}{\alpha _{\ell,m,n}})$ and $(\beta _{\ell,m,n}, \frac{1}{\beta_{\ell,m,n}})$ satisfy
 $$
0<\frac{1}{\alpha _{\ell,m,n}}<\frac{1}{\beta _{\ell,m,n}}<1<\beta _{\ell,m,n}<\alpha _{\ell,m,n}=\tau_{\ell,m,n},
 $$ 
where $\tau_{\ell,m,n}$ is the growth rate.
Furthermore, the sequence $\{ \tau _{\ell,m,n}\}$ converges to $\infty $ as $n \rightarrow \infty$.
 \item $Q_{\ell, m, n}(t)$ does not have a quadratic factor in $\Z[t]$.
\label{case1}
\end{enumerate}
\end{theo}

\begin{proof}
(1)\;We adapt Proposition \ref{root1} (cf. \cite[Theorem 2]{ZZ}) to $Q_{\ell,m,n}(t)$.
At first, $Q_{\ell,m,n}(\pm 1)\not=0$.
Consider the polynomial
$$
K_{\ell,m,n}(t):=(t-i)^{18}Q_{\ell,m,n}\left( \displaystyle{\frac{t+i}{t-i}} \right),
$$ and replace $u=t^2$. 
Then $K_{\ell, m, n}(u)$ can be written as follows.\\ \\
$ \displaystyle
     K_{\ell,m,n}(u)=  4\{  (8n+8)u^9+(147n+45m+30\ell+207)u^8
            -(3068n+360m+160\ell+3148)u^7+(11256n+364m-184\ell+7208)u^6
            -(10124n-616m-480\ell-6724)u^5-(7162n+722m-532\ell+32018)u^4
            +(12268n+40m-96\ell+27964)u^3-(4608n-428m+120\ell+8528)u^2 
            +(532n-168m+32\ell+836)u-(17n-13m+2\ell+21)\}            
$\\ \\
By considering the signs of $K_{\ell,m,n}(u)$ on the real line, it has seven positive real roots and two negative real roots:\\
\begin{tabular}{|c|c|c|c|c|c|c|c|c|c|c|}
\hline 
                          $u$&$-41$&$-31$&$0$&$1/10$&$1/3$&$1/2$&$1$&$2$&$3$&$9$\\
\hline                          
 sign$(K_{\ell,m,n}(u))$      &$-$   &$+$   &$-$   &$+$     &$-$     &$+$   &$-$ &$+$&$-$&$+$\\
\hline
\end{tabular}\\ \\
For example, by (\ref{ine}),
$$\frac{1}{4}K_{\ell,m,n}(0)\\
 =-21 - 2 \ell + 13 m - 17 n\\
 \leq -21 - 2 \ell + 13 n - 17 n\\
 =-21-2\ell-4n<0.$$
Then by Proposition \ref{root1}, $Q_{\ell,m,n}(t)$ has exactly four real roots, 
and all other roots are complex and on the unit circle.

(2)\;Observe that $Q_{\ell,m,n}(0)=1>0$, $Q_{\ell,m,n}(\frac{1}{2})<0$, and $Q_{\ell,m,n}(1)=32+32n>0$.
Hence
$$
0< \frac{1}{\alpha _{\ell,m,n}} <\frac{1}{2} < \frac{1}{\beta _{\ell,m,n}} <1< \beta _{\ell,m,n} < \alpha _{\ell,m,n}.
$$
Furthermore, $Q_{\ell,m,n}(\frac{1}{4n+5})<0$ and $Q_{\ell,m,n}(\frac{1}{4n+m+l+6})>0$, that is, 
$$
\frac{1}{4n+m+\ell+6}<\frac{1}{\alpha_{\ell,m,n}}=\frac{1}{\tau _{\ell,m,n}} < \frac{1}{4n+5},
$$
for all $n \geq 0$.
This implies $\tau_{\ell,m,n}\rightarrow \infty$ as $n \rightarrow \infty$.

(3)\;Let us fix $(\ell,m,n)$.
Assume that $Q_{\ell, m,n}(t)$ is written as
$$Q_{\ell,m,n}(t)=(1+at+t^2)(1+\sum _{k=1}^{8}b_kt^k+\sum _{k=1}^{7}b_{8-k}t^{k+8}+t^{16})$$ where $a$ and $b_k$ are integers.
Then by comparing the coefficients of both sides, we get the simultaneous equations for $a, b_1, b_2,..., b_8$:
\begin{equation}
\left\{
\begin{array}{l}
b_1=-a+(-6-4n) \\
b_2=-ab_1-1+(3-m+2n)\\
b_3=-ab_2-b_1+(-5 - \ell + m - 3 n)\\
b_4=-ab_3-b_2+(5 - 3 m + 5 n)\\
b_5=-ab_4-b_3+(-1 + 4 m - n)\\
b_6=-ab_5-b_4+(9 + \ell - 4 m + 8 n)\\
b_7=-ab_6-b_5+(-\ell + 5 m)\\
b_8=-ab_7-b_6+(11 + \ell - 5 m + 10 n)\\
ab_8+2b_7-(-2 + 6 m - 2 n)=0
\end{array}
\right.
\label{equations}
\end{equation}
By using the method of successive substitution inductively, the last equation of (\ref{equations}) is 
\begin{equation}
\begin{split}
f_{\ell,m,n}(a)
&:=a^9+(4n+6)a^8+(2n-m-6)a^7+(-29n-m+\ell-43)a^6+(-9n\\
&+4m+11)a^5+(63n+2m-6k+91)a^4+(11n-3m+\ell-4)a^3+(-41n\\
&+2m+10\ell-55)a^2+(-3n-m-2\ell-3)a+(6n-2m-4\ell+6)=0.
\end{split}
\label{integer1} 
\end{equation}
On the other hand, the signs of $f_{\ell,m,n}(t)$ on the real line are as follows:\\
\begin{tabular}{|c|c|c|c|c|c|c|c|c|c|c|c|}
\hline 
                               $t$&$-(4n+6)$&$-(4n+5)$&$-3$&$-2$&$-1$&$-1/2$&$0$&$1/2$&$1$&$8/5$&$2$\\
\hline                               
 sign$(f_{\ell,m,n}(t))$      &$-$   &$+$ &$+$  &$-$   &$+$     &$-$     &$+$   &$-$ &$+$&$-$&$+$\\
\hline
\end{tabular}\\ \\
For example, 
\begin{itemize}
 \item $f_{\ell,m,n}(-3)=5784 + 308 \ell + 748 m + 7368 n>0$,
 \item $f_{\ell,m,n}(-2)=-32 - 32 n<0$.
\end{itemize}
Hence $f_{\ell,m,n}(t)$ has one root in the open interval $(-(4n+6), -(4n+5))$, and has eight roots in the open interval $(-3,2)$ while $-2,-1,0,1$ are not roots of $f_{\ell,m,n}(t)$.
Therefore $f_{\ell,m,n}(t)=0$ cannot have an integer solution, which contradicts to (\ref{integer1}).
\end{proof}

Now we are interested in whether the growth rates $\tau_{\ell,m,n}$ are $2$-Salem numbers or not.
To show this, by means of Theorem \ref{poles} (1), it is sufficient to prove $Q_{\ell,m,n}(t)$ is irreducible over $\Z$.
It is already shown in Theorem \ref{poles} (3) that $Q_{\ell,m,n}(t)$ is not described as a product of two palindromic polynomials of degree two and sixteen.
For the irreducibility of $Q_{\ell,m,n}(t)$, the next proposition plays an important role in view of the main theorem. 

\begin{prop}
\label{reducible}
If $Q_{\ell,m,n}(t)$ is not irreducible, then $Q_{\ell,m,n}(t)$ is described as a product of two distinct monic palindromic polynomials in $\Z[t]$ of even degree.
\end{prop}

\begin{proof}
By Theorem \ref{poles} (1), $Q_{\ell,m,n}(t)$ can be written as follows:
\begin{equation*}
\begin{split}
Q_{\ell,m,n}(t)&=(t-\alpha_{\ell,m,n})(t-\frac{1}{\alpha_{\ell,m,n}})(t-\beta_{\ell,m,n})(t-\frac{1}{\beta_{\ell,m,n}})(t-\omega_1)(t-\overline{\omega _1})\\
&\cdots (t-\omega_7)(t-\overline{\omega _7}),
\end{split}
\end{equation*}
where $\alpha_{\ell,m,n}, \frac{1}{\alpha_{\ell,m,n}}$, $\beta_{\ell,m,n}, \frac{1}{\beta_{\ell,m,n}}$ are two pairs of real roots and $\omega_1, \overline{\omega _1}$, ..., $\omega_7, \overline{\omega _7}$ are seven pairs of complex roots lying on the unit circle, and each two roots which are an inversive  pair are algebraic conjugate to each other.
Hence if $Q_{\ell,m,n}(t)$ is not irreducible, each of its factors is of even degree, and the claim follows.
\end{proof}

It is not easy to examine the irreducibility for all $(\ell,m,n)$, but we have the following result under certain restrictions.

\begin{theo}
For $n\e 1$ (mod $3$), $Q_{0,n,n}(t)$ and $Q_{n,0,n}(t)$ are irreducible over $\Z$.
As a consequence, the growth rates $\tau _{0,n,n}$ and $\tau _{n,0,n}$ are $2$-Salem numbers.
\label{main}
\end{theo}

\begin{proof}
It is sufficient by Proposition \ref{reducible} to prove that $Q_{0,n,n}(t)$ and $Q_{n,0,n}(t)$ do not have a monic palindromic factor of degree 4, 6, or 8 for $n\e 1$ (mod 3). 
Since the method for each case is the same, we will explain the case of degree $4$, only.
 
Let us fix $(\ell,m,n)$.
Suppose that $Q_{\ell,m,n}(t)=(1+at+bt^2+at^3+t^4)(1+\sum _{k=1}^{7}c_kt^k+\sum _{k=1}^{6}c_{7-k}t^{k+7}+t^{14})$ where $a,b$ and $c_k$ are integers.
Then by comparing the coefficients of both sides, we get the simultaneous equations for $a, b, c_1, c_2,..., c_7$:
\begin{equation}
\left\{
\begin{array}{l}
c_1=-a+(-6-4n) \\
c_2=-ac_1-b+(3-m+2n)\\
c_3=-ac_2-bc_1-a+(-5 - \ell + m - 3 n)\\
c_4=-ac_3-bc_2-ac_1-1+(5 - 3 m + 5 n)\\
c_5=-ac_4-bc_3-ac_2-c_1+(-1 + 4 m - n)\\
c_6=-ac_5-bc_4-ac_3-c_2+(9 + \ell - 4 m + 8 n)\\
c_7=-ac_6-bc_5-ac_4-c_3+(-\ell + 5 m)\\
c_6=-ac_7-bc_6-ac_5-c_4+(11 + \ell - 5 m + 10 n)\\
c_5=-ac_6-bc_7-ac_6-c_5+(-2+6m-2n)
\end{array}
\right.
\label{equations2}
\end{equation}
By using the method of successive substitution inductively, the last two equations in (\ref{equations2})
are described as follows:
\begin{enumerate}
\item 
\label{1}
$
f_{\ell,m,n}(a,b)\\
:=-1 - a^8 - b^4 - m + a^7 (-6 - 4 n) + b^2 (1 + 2 m - 3 n) + 
 a^6 (-8 + 7 b + m - 2 n) + b (\ell + m - n) + n + b^3 (2 - m + 2 n) + 
 a^4 (-11 - 15 b^2 + 6 m - 11 n + b (30 - 5 m + 10 n)) + 
 a (15 + 2 \ell + 2 m + b (-46 - 8 m - 30 n) + 
    b^2 (-15 - 3 \ell + 3 m - 9 n) + 9 n + b^3 (24 + 16 n)) + 
 a^2 (2 + 10 b^3 - \ell + 3 m + b^2 (-24 + 6 m - 12 n) - 5 n + 
    b (9 - 12 m + 21 n)) + a^5 (-29 - \ell + m - 19 n + b (36 + 24 n)) + 
 a^3 (13 - 2 \ell + 6 m + b^2 (-60 - 40 n) + 9 n + 
    b (68 + 4 \ell - 4 m + 44 n))\\
=0$,
\item
\label{2}
$
g_{\ell,m,n}(a,b)\\
:=12 + a^7 (2 - b) + 2 m + b^2 (-23 - 4 m - 15 n) + 
 b^3 (-5 - \ell + m - 3 n) + 8 n + b^4 (6 + 4 n) + 
 a^5 (12 + 6 b^2 - 2 m + b (-18 + m - 2 n) + 4 n) + 
 a^6 (12 + b (-6 - 4 n) + 8 n) + b (15 + 2 \ell + 2 m + 9 n) + 
 a^3 (6 - 10 b^3 - 8 m + b (-35 + 12 m - 23 n) + 14 n + 
    b^2 (36 - 4 m + 8 n)) + 
 a (4 b^4 + 2 \ell + 2 m + b (4 - \ell + 7 m - 11 n) + 
    b^3 (-14 + 3 m - 6 n) - 2 n + b^2 (10 - 10 m + 18 n)) + 
 a^4 (34 + 2 \ell - 2 m + b (-77 - \ell + m - 51 n) + 22 n + 
    b^2 (30 + 20 n)) + 
 a^2 (-46 - 8 m + b^3 (-36 - 24 n) + b (-7 - 6 \ell + 10 m - 3 n) - 
    30 n + b^2 (87 + 3 \ell - 3 m + 57 n))\\
=0$.
\end{enumerate}

First, consider the case $(\ell,m,n)=(0,n,n)$, that is, $Q_{0,n,n}(t)$, which is the denominator of the growth function of the Coxeter polytope obtained by gluing together $n+1$ copies of $T$ only along the facets of type $B$ (cf. Figure \ref{dominoB}).
\begin{figure}[h]
\begin{center}
 \includegraphics [width=210pt, clip]{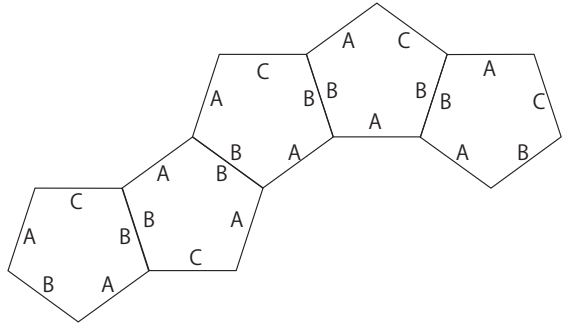}
\caption{The image of the Coxeter polytope obtained by $5$-times gluing of $T$ satisfying $(\ell, m, n)=(0,5,5)$}
\label{dominoB}
\end{center}
\end{figure}
By the assumption, there exist $a,b\in \Z$ which depend on $n$ and satisfy the two equations $(1)$ and $(2)$.
If $(a,b)\e (0,0)$ (mod 3), then $f_{0,n,n}(a,b)\e f_{0,n,n}(0,0)=-1$ (mod 3), which contradicts to (\ref{1}).
Hence $(a,b)\e (0,0)$ (mod 3) is impossible.
If $(a,b)\e (1,0)$ (mod 3), then $f_{0,n,n}(a,b)\e f_{0,n,n}(1,0)=-(26+4n)\e 1-n$ (mod 3).
So it is possible to satisfy (\ref{1}) only if  $n\e 1$ (mod 3).
On the other hand, $g_{0,n,n}(a,b)\e g_{0,n,n}(1,0)=32+8n\e -(1+n)$.
Hence it is possible to satisfy (\ref{2}) only if $n\e -1$ (mod 3), which is contradiction.
So $(a,b)\e (1,0)$ (mod 3) is also impossible.
The Table \ref{list1} below is the list of the values of $f_{0,n,n}(a,b)$ and $g_{0,n,n}(a,b)$, and possibility for $(a,b)$ to satisfy (\ref{1}) and (\ref{2}) for all the cases of $(a,b)$ modulo $3$. 
(We leave the box for the value of $g_{0,n,n}(a,b)$ empty if the value of $f_{0,n,n}(a,b)$ gives us  sufficient information.)
\begin{table}
\begin{tabular}{|c|c|c|c|} 
\hline 
$(a,b)$&$f_{0,n,n}(a,b)$&$g_{0,n,n}(a,b)$&\\
\hline
$(0,0)$&$-1$&$$&impossible\\
\hline 
$(1,0)$&$-26-4n\e 1-n$&$32+8n\e -(1+n)$&impossible\\
\hline 
$(-1,0)$&$-12-12n\e 0$&$-8-8n\e 1+n$&possible only if $n\e -1$\\
\hline 
$(0,1)$&$1$&$$&impossible\\
\hline 
$(0,-1)$&$-3-2n\e n$&$-15-14n\e n$&possible only if $n\e 0$\\
\hline 
$(1,1)$&$0$&$-2-2n\e 1+n$&possible only if $n\e -1$\\
\hline 
$(-1,1)$&$0$&$-6-4n\e -n$&possible only if $n\e 0$\\
\hline 
$(1,-1)$&$-280-114n\e -1$&$$&impossible\\
\hline 
$(-1,-1)$&$48+54n\e 0$&$78+82n\e n$&possible only if $n\e0$\\
\hline
\end{tabular}
\caption{The list of the values of $f_{0,n,n}(a,b)$ and $g_{0,n,n}(a,b)$ and possibility for $(a,b)$ to satisfy $f_{0,n,n}(a,b)=g_{0,n,n}(a,b)=0$ for all the cases of $(a,b)$ modulo $3$}
\label{list1}
\end{table}
Therefore there are no integers $a$ and $b$ satisfying $f_{0,n,n}(a,b)=g_{0,n,n}(a,b)=0$ if $n \e 1$ (mod 3), which implies that $Q_{0,n,n}(t)$ has no monic palindromic factor of degree $4$ if $n \e 1$ (mod 3). 

Next consider the case $(\ell,m,n)=(n,0,n)$, that is, $Q_{n,0,n}(t)$, which is the denominator of the growth function of the Coxeter polytope obtained by gluing together $n+1$ copies of $T$ only along the facets of type $A$ (cf. Figure \ref{dominoA}).
\begin{figure}[h]
\begin{center}
 \includegraphics [width=210pt, clip]{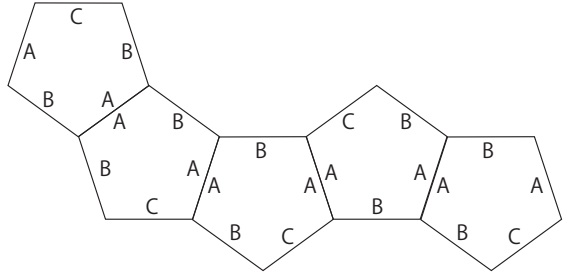}
\caption{The image of the Coxeter polytope obtained by $5$-times gluing of $T$ satisfying $(\ell, m, n)=(5,0,5)$}
\label{dominoA}
\end{center}
\end{figure}
The Table \ref{list2} is the list for all the cases of $(a,b)$ modulo $3$ as for the case $(l,m,n)=(0,n,n)$.
\begin{table}
\begin{tabular}{|c|c|c|c|} 
\hline 
$(a,b)$&$f_{n,0,n}(a,b)$&$g_{n,0,n}(a,b)$& \\
\hline
$(0,0)$&$-1+n$&$12+8n\e -n$&impossible\\
\hline 
$(1,0)$&$-26-24n\e 1$&&impossible\\
\hline 
$(-1,0)$&$-12-12n\e 0$&$-8-8n\e 1+n$&possible only if $n\equiv -1$\\
\hline 
$(0,1)$&$1$&$$&impossible\\
\hline 
$(0,-1)$&$-3-4n\e -n$&$-15-10n\e -n$&possible only if $n\e 0$\\
\hline 
$(1,1)$&$0$&$-2-2n\e 1+n$&possible only if $n\e -1$\\
\hline 
$(-1,1)$&$0$&$-6-4n\e -n$&possible only if $n\e 0$\\
\hline 
$(1,-1)$&$-280-182n\e -1+n$&$378+248n\e -n$&impossible\\
\hline 
$(-1,-1)$&$48+50n\e -n$&$78+74n\e -n$&possible only if $n\e 0$\\
\hline
\end{tabular}
\caption{The list of the values of $f_{n,0,n}(a,b)$ and $g_{n,0,n}(a,b)$ and possibility for $(a,b)$ to satisfy $f_{n,0,n}(a,b)=g_{n,0,n}(a,b)=0$ for all the cases of $(a,b)$ modulo $3$}
\label{list2}
\end{table}
Therefore there are no integers $a$ and $b$ satisfying $f_{n,0,n}(a,b)=g_{n,0,n}(a,b)=0$ if $n \e 1$ (mod 3), which implies that $Q_{n,0,n}(t)$ has no monic palindromic factor of degree $4$ if $n \e 1$ (mod 3). 
\end{proof}





\end{document}